\documentclass[12 pt]{amsart}
\usepackage[all]{xy}

\usepackage[urlcolor=blue]{hyperref} 

\usepackage{enumitem} 
\usepackage{xspace} 

\usepackage{amsfonts}
\usepackage{amsmath,amssymb}      
\usepackage{amsthm}
\usepackage{pdfsync}       
\usepackage{color}
\usepackage[english]{babel}

\usepackage[T1]{fontenc}

\usepackage{color}

\usepackage{tikz}
\tikzset{
v/.style={
  circle, draw, inner sep=2pt, minimum size=6pt, fill=white},
l/.style={
  circle, draw, inner sep=2pt, minimum size=6pt, fill=black}
}

\theoremstyle{plain}
\newtheorem{theorem}{Theorem}[section]
\newtheorem{proposition}[theorem]{Proposition}
\newtheorem{lemma}[theorem]{Lemma}
\newtheorem{corollary}[theorem]{Corollary}

\theoremstyle{definition}
\newtheorem{example}[theorem]{Example}
\newtheorem{definition}[theorem]{Definition}
\newtheorem{remark}[theorem]{Remark}

 \newcommand \C {{\mathbb C}}
 \newcommand \A{{\mathcal A}}

\newcommand \spann  {{\rm span}}

\def \X(#1){\{x_1,\dots, x_{#1}\}}

\begin{document}

\title {On the Falk invariant of signed graphic arrangements}
\begin{abstract} The fundamental group of the complement of a hyperplane arrangement in a complex vector space is an important topological invariant. The third rank of successive quotients in the lower central series of the fundamental group was called \emph{Falk invariant} of the arrangement since Falk gave the first formula and asked to give a combinatorial interpretation. In this article, we give a combinatorial formula for the Falk invariant of a signed graphic arrangement that do not have a $B_2$ as sub-arrangement.
\end{abstract}

\author{Weili Guo}
\address{Department of Mathematics, Hokkaido University, Kita 10, Nishi 8, Kita-Ku, Sapporo 060-0810, Japan.}
\email{guowl@math.sci.hokudai.ac.jp}
\author{Michele Torielli}
\address{Department of Mathematics, Hokkaido University, Kita 10, Nishi 8, Kita-Ku, Sapporo 060-0810, Japan.}
\email{torielli@math.sci.hokudai.ac.jp}

\date{\today}
\maketitle


\section{Introduction}
A \textbf{hyperplane} $H$ in $\C^\ell$ is an affine subspace of dimension $\ell-1$. A finite collection $\A = \{H_1, \dots , H_n\}$ of hyperplanes is called a \textbf{hyperplane arrangement}. If $\bigcap_{i=1}^{n}H_i\ne\emptyset$, then $\A$ is called \textbf{central}. In this paper, we only consider central arrangements and assume that all the hyperplanes contain the origin. For more details on hyperplane arrangements, see \cite{orlik2013arrangements}.

Let $M:=\C^\ell\setminus_{H\in\A}H$ be the complement of the arrangement $\A$. It is known that the cohomology ring $H^*(M)$ is completely determined by $L(\A)$ the lattice of intersection of $\A$. Similarly to this result, there are several conjectures concerning the relationship between $M$ and $L(\A)$. To study such problems, Falk introduced in \cite{falk1990algebra} a multiplicative invariant, called \textbf{global invariant}, of the Orlik-Solomon algebra of $\A$. The invariant is now known as the ($3^{rd}$) \textbf{Falk invariant} and it is denoted by $\phi_3$. In \cite{falk2001combinatorial}, Falk posed as an open problem to give a combinatorial interpretation of $\phi_3$. 

Several authors already studied this invariant. In \cite{schenck2002lower}, Schenck and Suciu studied the lower central series of arrangements and described a formula for the Falk invariant in the case of graphic arrangements. In \cite{guo2017global}, the authors gave a formula for $\phi_3$ in the case of simple sign graphic arrangements. In the preprint  \cite{guo2017falkinvar}, the authors extended the previous result for sign graphic arrangements coming from graphs without loops. This article is devoted to extend these results further and to describe a combinatorial formula for the Falk invariant  of a signed graphic arrangement that do not have a $B_2$ as sub-arrangement. Our result gives a partial answer to the question posed by Falk in \cite{falk2001combinatorial}.

\section{Preliminares on Orlik-Solomon algebras}
Let $\A=\{H_1, \dots, H_n\}$ be an arrangement of hyperplanes in $\C^{\ell}$.
Let $E^1=\bigoplus_{j=1}^n\C e_j$ be the free module generated by $e_1, e_2, \dots, e_n$, where $e_i$ is a symbol corresponding to the hyperplane $H_i$.
Let $E=\bigwedge E^1$ be the exterior algebra over $\C$. The algebra $E$ is graded via $E=\bigoplus_{p=0}^nE^p$, where $E^p=\bigwedge^pE^1$.
The $\C$-module $E^p$ is free and has the distinguished basis consisting of monomials $e_S:=e_{i_1}\wedge\cdots\wedge e_{i_p}$,
where $S=\{{i_1},\dots, {i_p}\}$ is running through all the subsets of $\{1,\dots,n\}$ of cardinality $p$ and $i_1<i_2<\cdots<i_p$.
The graded algebra $E$ is a commutative DGA with respect to the differential $\partial$ of degree $-1$ uniquely defined
by the conditions $\partial e_i=1$ for all $i=1,\dots, n$ and the graded Leibniz formula. Then for every $S\subseteq\{1,\dots,n\}$ of cardinality $p$
$$\partial e_S=\sum_{j=1}^p(-1)^{j-1}e_{S_j},$$
where $S_j$ is the complement in $S$ to its $j$-th element.

For every $S\subseteq\{1,\dots,n\}$, put $\cap S=\bigcap_{i\in S}H_i$ (possibly $\cap S=\emptyset$). The set of all
intersections $L(\A):=\{\cap S\mid S\subseteq \{1,\dots,n\}\}$ is called the \textbf{intersection poset of $\A$}.
The subset $S\subseteq\{1,\dots,n\}$ is called \textbf{dependent} if $\cap S\ne\emptyset$
and the set of linear polynomials $\{\alpha_i~|~i\in S\}$ with $H_i=\alpha_i^{-1}(0)$, is linearly dependent.
\begin{definition}\label{def:osalgbr}
The \textbf{Orlik-Solomon ideal} of $\A$ is the ideal $I=I(\A)$ of $E$ generated by
\begin{enumerate}
\item[$(1)$] all $e_S$ with $\cap S=\emptyset$,
\item[$(2)$] all $\partial e_S$ with $S$ dependent.
\end{enumerate}
The algebra $A:=A^\bullet(\A)=E/I(\A)$ is called the \textbf{Orlik-Solomon algebra} of $\A$.
\end{definition}
Clearly $I$ is a homogeneous ideal of $E$ and $I^p=I\cap E^p$ whence $A$ is a graded algebra and we can write $A=\bigoplus_{p\ge 0} A^p$, where $A^p=E^p/I^p$.
If $\A$ is central, then for any $S\subseteq\A$, we have $\cap S\neq\emptyset$. Therefore, the Orlik-Solomon ideal is generated
by the elements of type $(2)$ from Definition \ref{def:osalgbr}.
In this case, the map $\partial$ induces a well-defined differential $\partial\colon A^\bullet(\A)\longrightarrow A^{\bullet -1}(\A)$.

Let $I_k$ be the $k$-adic Orlik-Solomon ideal of $\A$ generated by $\sum_{j\le k}I^j$ in $E$. It is clear that $I^k$ is a graded ideal and $I_k^p = (I_k)^p = E^p\cap I_k$. Write $A_k:=A^\bullet_k(\A) = E/I_k$ and
$A_k^p:= (A^\bullet_k(\A))^p=E^p/I_k^p$ which is called \textbf{$k$-adic Orlik-Solomon algebra} by Falk \cite{falk1990algebra}.


In this set up, it is now easy to define the Falk invariant.
\begin{definition} Consider the map $d$ defined by
$$d\colon E^1\otimes I^2\to E^3,$$
$$d(a\otimes b)=a\wedge b.$$
Then the \textbf{Falk invariant} is defined as
$$\phi_3:=\dim(\ker(d)).$$
\end{definition}

In \cite{falk1990algebra} and \cite{falk2001combinatorial}, Falk gave a beautiful formula to compute such invariant.
\begin{theorem}[Theorem 4.7, \cite{falk2001combinatorial}]\label{theo:falkinvar} Let $\A=\{H_1, \dots, H_n\}$ be an arrangement of hyperplanes in $\C^{\ell}$. Then
\begin{equation}\label{eq:falktheorem1}
\phi_3=2\binom{n+1}{3}-n\dim(A^2)+\dim(A^3_2).
\end{equation}
\end{theorem}
\begin{remark}\label{rem:falkinvariantreduct} Since $\dim(A^3_2)=\dim((E/I_2)^3)=\dim(E^3)-\dim(I^3_2)$ and $\dim(E^3)=\binom{n}{3}$, then we obtain
\begin{equation}\label{eq:falktheorem}
\phi_3=2\binom{n+1}{3}-n\dim(A^2)+\binom{n}{3}-\dim(I^3_2).
\end{equation}
\end{remark}

Recall that $\phi_3$ can also be describe from the lower central series of the fundamental group $\pi(M)$ of the complement $M$ of the arrangement. In particular, if we consider the lower central series as a chain of normal subgroups $N_i$, for $k \ge 1$, where $N_1 = \pi(M)$ and $N_{k+1} = [N_k,N_1]$, the subgroup generated by commutators of elements in $N_k$ and $N_1$, then $\phi_3$ is the rank of the finitely generated abelian group $N_3/N_4$. See \cite{schenck2002lower} for more details.

\section{Sign graphs}
In this section we will recall the main properties of signed graphs. See \cite{zaslavsky1982signed} for a general treatment of such graphs.
\begin{definition} A \textbf{signed graph} is a tuple $ G=(V_{G},E_{G}^{+},E_{G}^{-},L_{G}) $, where
\begin{itemize}
\item $ V_{G} $ is a finite set called the set of vertices,
\item $ E_{G}^{+} $ is a subset of $ \binom{V_{G}}{2} $ called the set of positive edges,
\item $ E_{G}^{-} $ is a subset of $ \binom{V_{G}}{2} $ called the set of negative edges,
\item $ L_{G} $ is a subset of $ V_{G} $ called the set of loops.
\end{itemize}
\end{definition}

\begin{example}
In this article, we illustrate a signed graph as follows:
\begin{align*}
G = (V_{G}, E_G^{+}, E_G^{-}, L_G) =
\begin{tikzpicture}[baseline=10pt]
\draw (0,1) node[v, label=above:{1}](1){};
\draw (0,0) node[v, label=below:{2}](2){};
\draw (1,0) node[l, label=below:{3}](3){};
\draw (1,1) node[l, label=above:{4}](4){};
\draw[] (1)--(2);
\draw[double, thick] (1)--(3);
\draw [] (1)--(4);
\draw[dashed] (2)--(3);
\draw[dashed] (2)--(4);
\end{tikzpicture}, \qquad
\begin{cases}
V_{G} = \{1,2,3,4\}, \\
E_G^{+} = \{\{1,2\}, \{1,3\}, \{1,4\}\}, \\
E_G^{-} = \{\{1,3\}, \{2,3\},\{2,4\}\}, \\
L _G= \{3,4\}.
\end{cases}
\end{align*}
\end{example}
Let $ G^{+}=(V_{G}, E_{G}^{+}) $ and $ G^{-}=(V_{G},E_{G}^{-}) $, then
we have an alternative notation $G=(G^{+},G^{-},L_{G})$ for the signed graph $G$.
An unsigned simple graph $G$ may be regarded as a signed graph $ G=(G,K_\ell^\circ,\emptyset) $, where $K_{\ell}^\circ$ denotes the edgeless graph on $\ell$ vertices.
A signed graph $ (G^{+},G^{-},\emptyset) $ is called \textbf{loopless}, which is also denoted by $ (G^{+}, G^{-}) $.
Let $ E_{G} $ denote the edge set $ E_{G}^{+} \sqcup E_{G}^{-} \sqcup L_{G} $.
For a positive integer $ \ell $, let $ [\ell] $ denote the set $ \{1, \dots, \ell \} $.
From now on, we suppose that $ G $ is a signed graph on vertices $ [\ell] $.
Let $ (x_{1}, \dots, x_{\ell}) $  be a basis for the $ \ell $-dimensional vector space $ (\C^{\ell})^{\ast} $.
For $ \alpha \in (\C^{\ell})^{\ast} $, let $ \{\alpha=0\} $ denote the hyperplane $ \{ v \in \C^{\ell} ~|~ \alpha(v)=0 \} $.
\begin{definition}
Given a signed graph $G$, let $\mathcal{A}(G)$ be the hyperplane arrangement in $\C^{\ell} $ consisting of the following hyperplane
$$\{x_{i}-x_{j}=0\} \text{ for }  \{i,j\} \in E_{G}^{+},$$
$$\{x_{i}+x_{j}=0\} \text{ for } \{i,j\} \in E_{G}^{-},$$
$$\{x_{i}=0\} \text{ for } i \in L_{G}.$$
We will call $\mathcal{A}(G)$ the \textbf{signed graphic arrangement} associated to the signed graph $G$.
\end{definition}

Given a signed graph it is natural to introduce the following function.
\begin{definition} Given a sign graph $G=(V_{G},E_{G}^{+},E_{G}^{-},L_{G})$, the \textbf{sign function} of $G$ is the function $sgn\colon E_{G}^{+}\cup E_{G}^{-}\cup L_G\to\{+,-\}$ defined by
$$sgn(e)=\left\{
    \begin{array}{rl}
      + & \text{if } e\in E_{G}^{+},\\
      - & \text{if } e\in E_{G}^{-}\cup L_G.
    \end{array} \right.$$
\end{definition}
We can naturally extend the previous definition to path in $G$
\begin{definition} Given $P=e_1e_2\cdots e_k$ a path in $G$, the \textbf{sign} of $P$ is $sgn(P)=\prod_{i=1}^ksgn(e_i)$.
\end{definition}
\begin{definition} A cycle $C$ in a sign graph $G$ is called \textbf{balanced} if $sgn(C)=+$.
\end{definition}
Given a sign graph $G$ and a function $\sigma\colon V_G\to\{+,-\}$, we can define a new sign graph $G'$ that has the same underlying graph as $G$ but with a different sign function. In particular, if $e=\{i,j\}\in E_G$ then $sgn_{G'}(e)=\sigma(i)sgn_G(e)\sigma(j)$.
\begin{definition} In the previous construction, we will call $G'$ the \textbf{switching of $G$ by $\sigma$} and we will denote it by $G^\sigma$. In this case, $\sigma$ is called a \textbf{switching function for $G$}.
\end{definition}
\begin{definition} Given two sign graph $G_1$ and $G_2$ with the same underlying graph, we will say they are \textbf{switching equivalent} and write $G_1\backsim G_2$, if there exists a switching function $\sigma$ such that $G_2=G_1^\sigma$.
\end{definition}
\begin{proposition}[Proposition 3.2, \cite{zaslavsky1982signed}] Two signed graphs with the same underlying graph are switching equivalent if and only if they have the same list of balanced circles.
\end{proposition}
\begin{proposition}[Corollary 5.4, \cite{zaslavsky1982signed}] Two signed graphs with the same underlying graph are switching equivalent if and only if they define the same matroid.
\end{proposition}
Using the previous results, we obtain
\begin{corollary}\label{corol:signequivsameinvar} Let $G_1$ and $G_2$ be two signed graph with the same underlying graph. If $G_1\backsim G_2$, then $\phi_3(\A(G_1))=\phi_3(\A(G_2)).$
\end{corollary}
In this paper taking inspiration from graph theory and the study of hyperplane arrangements, we denote by $K_\ell$ a complete graph with $\ell$ vertices and all edges being positive, i.e. $K_\ell=(K_\ell, K_\ell^\circ,\emptyset)$, by $D_\ell$ a complete sign graph with $\ell$ vertices and no loops, i.e. $D_\ell=(K_\ell, K_\ell,\emptyset)$, and by $B_\ell$ a sign complete graph with $\ell$ vertices and a full set of loops, i.e. $B_\ell=(K_\ell, K_\ell, [\ell])$. Moreover, we denote by $K_\ell^\ell$ a complete graph with $\ell$ vertices, all edges being positive and a full set of loops, i.e. $K_\ell^\ell=(K_\ell, K_\ell^\circ, [\ell])$, by $D_\ell^1$ a complete sign graph with $\ell$ vertices and one loop, i.e. $D_\ell^1=(K_\ell, K_\ell,\{1\})$ and by $G_\circ$ the signed graph in Figure \ref{fig:double1loop}. Furthermore, if $G$ is a signed graph we denote but $\overline{G}$ a signed graph switching equivalent to $G$ for some switching function $\sigma$.

\begin{figure}[htbp]
\begin{tikzpicture}[baseline=10pt]
\draw (0,0) node[v](1){};
\draw (2,3) node[l](2){};
\draw (4,0) node[v](3){};
\draw[double, thick] (1)--(2);
\draw[double, thick] (2)--(3);
\draw [] (1)--(3);
\end{tikzpicture}
\caption{The sign graph $G_\circ$}
\label{fig:double1loop}
\end{figure}
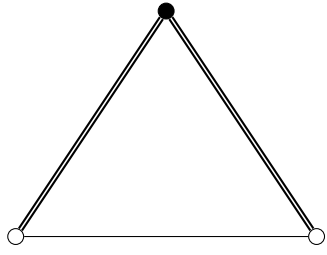

\section{Main Theorem}
In this section we describe how to compute the Falk invariant $\phi_3$ for $\mathcal{A}(G)$, a signed graphic arrangement associated to a signed graph $G$ that do not have  a subgraph isomorphic to $B_2$. In the remaining of the paper, to fix the notation we will suppose $G$ is a graph on $\ell$ vertices having $n$ edges, and we will label only the edges as elements of $[n]:=\{1,\dots, n\}$.

The goal of this section is to prove the following theorem.
\begin{theorem}\label{theo:ourmain} For a signed graphic arrangement associated to a signed graph $G$ not containing a subgraph isomorphic to $B_2$ as subgraph, we have
\begin{equation}\label{eq:ourmainformula}
\phi_3=2(k_3+k_4+d_3+d_{2,1}+k_{2,2}+k_{3,3}+g_\circ)+5d_{3,1},
\end{equation}
where $k_l$ denotes the number of subgraph of $G$ isomorphic to a $\overline{K_l}$, $d_l$ denotes the number of subgraph of $G$ isomorphic to $D_l$ but not contained in $D_l^1$, $d_{l,1}$ denotes the number of subgraph of $G$ isomorphic to $D_l^1$, $k_{l,l}$ denotes the number of subgraph of $G$ isomorphic to a $\overline{K_l^l}$ and $g_\circ$ denotes the number of subgraph of $G$ isomorphic to a $\overline{G_\circ}$ but not contained in $D_l^1$.
\end{theorem}

In order to compute $\phi_3$, we need firstly to identify the ordered $3$-tuple $S$ in $\{1, \dots, n\}$ that are dependent. Clearly, we have the following
\begin{lemma} $S=(i_1, i_2, i_3)$ is dependent if and only if $i_1, i_2, i_3$ correspond to the edges of a subgraph of $G$ that is isomorphic to a $\overline{K_3}$, or a $D_2^1$ or a $\overline{K_2^2}$.
\end{lemma}
With an abuse of notation, we will call a dependent $3$-tuple $S$ a \textbf{triangle}. Moreover, we will write
$$\mathcal{C}_3:=\{e_S\in E~|~S \text{ is a triangle}\}$$
which is a subset of $E$ as a vector space over $\C$.

\begin{remark} Notice that the triangles are exactly the balanced $3$-cycles together with the subgraphs isomorphic to $\overline{K_2^2}$. In particular, If $G_1$ and $G_2$ are two signed graph with the same underlying graph such that  $G_1\backsim G_2$, then $\mathcal{C}_3(G_1)=\mathcal{C}_3(G_2)$.
\end{remark}

Since $e_ie_je_k=-e_je_ie_k$, it is clear that the dimension of the vector space $\mathcal{C}_3$ is $k_3+d_{2,1}+k_{2,2}$. Moreover, we can consider $C'_3$ a basis of $\mathcal{C}_3$. Then each element of $C'_3$ is in a one-to-one correspondence of the subgraph of $G$ isomorphic to a $\overline{K_3}$, or a $D_2^1$ or a $\overline{K_2^2}$.

\begin{lemma}\label{lem:dimA2} $\dim(A^2)=\binom{n}{2}-k_3-d_{2,1}-k_{2,2}$.
\end{lemma}
\begin{proof} By definition $A=E/I$, hence 
$$\dim(A^2)=\dim(E^2)-\dim(I^2)=\binom{n}{2}-\dim(I^2).$$
Since $I^2=\spann\{\partial e_{ijk}~|~e_{ijk}\in\mathcal{C}_3\}$, then $\dim(I^2)=k_3+d_{2,1}+k_{2,2}$, and the thesis follows.
\end{proof}
Using Theorem \ref{theo:falkinvar} and Remark \ref{rem:falkinvariantreduct}, to prove Theorem \ref{theo:ourmain} we just need to describe $\dim(I^3_2)$. To do so, consider 
$$C_3:=\{e_t\partial e_{ijk}~|~e_{ijk}\in C'_3,t\in\{i,j,k\}\},$$
and
$$F_3:=\{e_t\partial e_{ijk}~|~e_{ijk}\in C'_3,t\in[n]\setminus\{i,j,k\}\}.$$

By construction $I^3_2=I^2\cdot E^1=\spann\{e_t\partial e_{ijk}~|~e_{ijk}\in C'_3,t\in[n]\}$, and hence
$$I^3_2=\spann(C_3)+\spann(F_3).$$
\begin{lemma} For a signed graphic arrangement associated to a signed graph $G$ not containing a subgraph isomorphic to $B_2$ as subgraph, we have
$$I^3_2=\spann(C_3)\oplus\spann(F_3).$$
\end{lemma}
\begin{proof} Since $G$ do not contain a $B_2$ as subgraph, any two triangles shares at most one element. This then gives us that $\spann(C_3)\cap\spann(F_3)=\emptyset$.
\end{proof}
\begin{remark} Notice that if we allow $G$ to have subgraphs isomorphic to $B_2$, then the previous lemma is not true anymore.
\end{remark}
By the previous lemma, we can write
$$\dim(I^3_2)=\dim(\spann(C_3))+\dim(\spann(F_3))=k_3+d_{2,1}+k_{2,2}+\dim(\spann(F_3)).$$
To prove our main result we need to be able to compute $\dim(\spann(F_3))$. To do so, consider the following sets
$$F^1_3:=\{e_t\partial e_{ijk}~|~e_{ijk}\in C'_3,t\in[n]\setminus\{i,j,k\}, i,j,k \text{ are not in the same } \overline{K_4}, D_3, \overline{G_\circ}, D_3^1, \overline{K_3^3}\},$$
$$F^2_3:=\{e_t\partial e_{ijk}~|~e_{ijk}\in C'_3,t\in[n]\setminus\{i,j,k\}, i,j,k \text{ are in the same } \overline{K_4}\},$$
$$F^3_3:=\{e_t\partial e_{ijk}~|~e_{ijk}\in C'_3,t\in[n]\setminus\{i,j,k\}, i,j,k \text{ are in the same } D_3 \text{ but not same } D_3^1\},$$
$$F^4_3:=\{e_t\partial e_{ijk}~|~e_{ijk}\in C'_3,t\in[n]\setminus\{i,j,k\}, i,j,k \text{ are in the same } \overline{G_\circ} \text{ but not same } D_3^1\},$$
$$F^5_3:=\{e_t\partial e_{ijk}~|~e_{ijk}\in C'_3,t\in[n]\setminus\{i,j,k\}, i,j,k \text{ are in the same } D_3^1\},$$
$$F^6_3:=\{e_t\partial e_{ijk}~|~e_{ijk}\in C'_3,t\in[n]\setminus\{i,j,k\}, i,j,k \text{ are in the same } \overline{K_3^3}\},$$

\begin{lemma} For a signed graphic arrangement associated to a signed graph $G$ not containing a subgraph isomorphic to $B_2$, we have
$$\spann(F_3)=\bigoplus_{i=1}^6\spann(F^i_3).$$
\end{lemma}
\begin{proof} Clearly, since $G$ does not contain a subgraph isomorphic to $B_2$, by construction $\spann(F^i_3)\cap\spann(F^j_3)=\emptyset$ for all $i,j=2, \dots, 6$ such that $i\ne j$.

For any element $e_t\partial e_{ijk}$ of $F^1_3$, we assert that at least one of the terms $e_{tjk}, e_{tik}, e_{tij}$ appears only in the expression of $e_t\partial e_{ijk}$. So $e_t\partial e_{ijk}$ can not be expressed linearly by the elements of $F^2_3, \dots, F^6_3$.

Since the edges $t, i, j, k$ are not in the same $\overline{K_4}, D_3, \overline{G_\circ}, D_3^1, \overline{K_3^3}$, and we do not consider the graphs having subgraphs isomorphic to $B_2$, we should consider three cases about the edge $t$: it can be adjacent to none of the edges $i,j,k$, to two of them, or to all of them.

Assume that the edge $t$ is adjacent to none of the edges $i,j,k$. This implies that $t$ and none of $i,j,k$ can appear in the same triangle. Hence any element $e_t\partial e_{ijk}$ of $F^1_3$ will not appear in any of $F^2_3, \dots, F^6_3$.

Assume now that the edge $t$ is adjacent to two of the edges $i,j,k$, then we should consider several possibilities.
Suppose that in the set $\{t,i,j,k\}$ there is no loop. If all the terms of the element $e_t\partial e_{ijk}\in F^1_3$ appear in $F^2_3, \dots, F^6_3$, then $t,i,j,k$ have to appear in the same $K_4$, but this is impossible by construction.
Suppose that $t$ is a loop and there is no loop in the set $\{i,j,k\}$. If all the terms of the element $e_t\partial e_{ijk}\in F^1_3$ appear in $F^2_3, \dots, F^6_3$, then $t,i,j,k$ have to appear in the same $\overline{G_\circ}$ or in the same $D_3^1$, but this is impossible by construction.
Suppose that $t$ is not a loop and there is one loop in the set $\{i,j,k\}$. In this case $i,j,k$ are the edges of a $D_2^1$. Hence, by assumption, the edges $t$ is not adjacent to the loop. If all the terms of the element $e_t\partial e_{ijk}\in F^1_3$ appear in $F^2_3, \dots, F^6_3$, then, also in this case, $t,i,j,k$ have to appear in the same $\overline{G_\circ}$ or in the same $D_3^1$, but this is impossible by construction.
Suppose that $t$ is not a loop and there are two loops in the set $\{i,j,k\}$. In this case $i,j,k$ are the edges of a $\overline{K_2^2}$. If all the terms of the element$e_t\partial e_{ijk}\in F^1_3$ appear in $F^2_3, \dots, F^6_3$, then $t,i,j,k$ have to appear in the same  $\overline{K_3^3}$, but this is impossible by construction.

Finally, assume that the edge $t$ is adjacent to all the edges $i,j,k$. In this situation, there are just two cases we should consider.
Suppose that in the set $\{t,i,j,k\}$ there is no loop. If all the terms of the element $e_t\partial e_{ijk}\in F^1_3$ appear in $F^2_3, \dots, F^6_3$, then $t,i,j,k$ have to appear in the same $D_3$, but this is impossible by construction.
Suppose that $t$ is not a loop and there is one loop in the set $\{i,j,k\}$. In this case $i,j,k$ are the edges of a $D_2^1$. If all the terms of the element $e_t\partial e_{ijk}\in F^1_3$ appear in $F^2_3, \dots, F^6_3$, then $t,i,j,k$ have to appear in the same $\overline{G_\circ}$ or in the same $D_3^1$, but this is impossible by construction. 

Therefore, for any element $e_t\partial e_{ijk}\in F^1_3$, at least one of the terms $e_{tjk}, e_{tik}, e_{tij}$ appears only in the expression of $e_t\partial e_{ijk}$. This shows that $\spann(F^i_3)\cap\spann(F^j_3)=\emptyset$ for all $i\ne j$. Since clearly
$$\spann(F_3)=\sum_{i=1}^6\spann(F^i_3)$$
this concludes the proof.
\end{proof}

%
%

\begin{example}\label{ex:G0comput} We consider the dimension of $\spann(F_3)$ for the sign graphic arrangement $A_3$ associated to the graph $G_\circ$ (see Figure \ref{fig:double1loop2}).

\begin{figure}[htbp]
\begin{tikzpicture}[baseline=10pt]
\draw (0,0) node[v](1){};
\draw (2,3) node[l, label=above:{6}](2){};
\draw (4,0) node[v](3){};
\draw[double, thick] (1)--(2);
\draw[double, thick] (2)--(3);
\draw[] (1)--(3);
\draw (1.3, 1.5) node {1};
\draw (0.9, 1.8) node {4};
\draw (2.7, 1.5) node {2};
\draw (3.1, 1.8) node {5};
\draw (2, 0.3) node {3};
\end{tikzpicture}
\caption{The sign graph $G_\circ$}
\label{fig:double1loop2}
\end{figure}
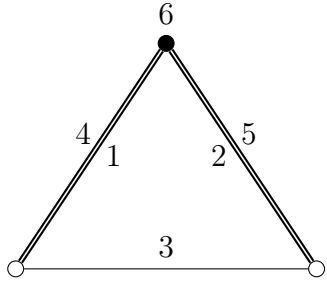

In this situation we have $E^+=\{1,2,3\}, E^-=\{4,5\}$ and $L=\{6\}$. Then
the number of the elements in $F_3$ is $12$, listed as follows.
$$ e_4 \partial e_{123}=e_{234}-e_{134}+e_{124}, e_5 \partial e_{123}=e_{235}-e_{135}+e_{125},$$
$$e_6 \partial e_{123}=e_{236}-e_{136}+e_{126}, e_1 \partial e_{345}=e_{145}-e_{135}+e_{134},$$
$$e_2 \partial e_{345}=e_{245}-e_{235}+e_{234}, e_6 \partial e_{345}=e_{456}-e_{356}+e_{346},$$
$$e_2 \partial e_{146}=e_{246}+e_{126}-e_{124}, e_3 \partial e_{146}=e_{346}+e_{136}-e_{136},$$
$$e_5 \partial e_{146}=-e_{456}+e_{156}+e_{145}, e_1 \partial e_{256}=e_{156}-e_{126}+e_{125},$$
$$e_3 \partial e_{256}=e_{356}+e_{236}-e_{235}, e_4 \partial e_{256}=e_{456}+e_{246}-e_{245}.$$
Then an easy computation shows that in this case $\dim(\spann(F_3))=10$. 
\end{example}

\begin{example}\label{exd31comput} We consider the dimension of $\spann(F_3)$ for the sign graphic arrangement associated to the graph $D_3^1$ (see Figure \ref{fig:d31graph}).

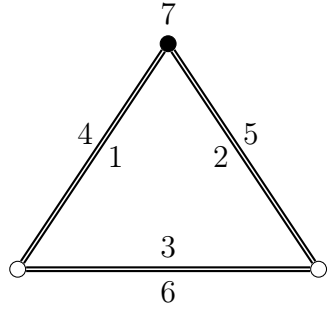
\begin{figure}[htbp]
\begin{tikzpicture}[baseline=10pt]
\draw (0,0) node[v](1){};
\draw (2,3) node[l, label=above:{7}](2){};
\draw (4,0) node[v](3){};
\draw[double, thick] (1)--(2);
\draw[double, thick] (2)--(3);
\draw[double, thick] (1)--(3);
\draw (1.3, 1.5) node {1};
\draw (0.9, 1.8) node {4};
\draw (2.7, 1.5) node {2};
\draw (3.1, 1.8) node {5};
\draw (2, 0.3) node {3};
\draw (2, -0.3) node {6};
\end{tikzpicture}
\caption{The sign graph $D_3^1$}
\label{fig:d31graph}
\end{figure}
In this situation we have $E^+=\{1,2,3\}, E^-=\{4,5,6\}$ and $L=\{7\}$. Then
the number of the elements in $F_3$ is $24$, listed as follows.
$$e_4 \partial e_{123}=e_{124}-e_{134}+e_{234}, e_5 \partial e_{123}=e_{125}-e_{135}+e_{235},$$
$$e_6 \partial e_{123}=e_{126}-e_{136}+e_{236}, e_7 \partial e_{123}=e_{127}-e_{137}+e_{237},$$
$$e_2 \partial e_{156}=-e_{125}+e_{126}+e_{256}, e_3 \partial e_{156}=-e_{135}+e_{136}+e_{356},$$
$$e_4 \partial e_{156}=-e_{145}+e_{146}+e_{456}, e_7 \partial e_{156}=e_{157}-e_{167}+e_{567},$$
$$e_1 \partial e_{246}=e_{124}-e_{126}+e_{146}, e_3 \partial e_{246}=-e_{234}+e_{236}+e_{346},$$
$$e_5 \partial e_{246}=e_{245}+e_{256}-e_{456}, e_7 \partial e_{246}=e_{247}-e_{267}+e_{467},$$
$$e_1 \partial e_{345}=e_{134}-e_{135}+e_{145}, e_2 \partial e_{345}=e_{234}-e_{235}+e_{245},$$
$$e_6 \partial e_{345}=e_{346}-e_{356}+e_{456}, e_7 \partial e_{345}=e_{347}-e_{357}+e_{457},$$
$$e_2 \partial e_{147}=-e_{124}+e_{127}+e_{247}, e_3 \partial e_{147}=-e_{134}+e_{137}+e_{347},$$
$$e_5 \partial e_{147}=e_{145}+e_{157}-e_{457}, e_6 \partial e_{147}=e_{146}+e_{167}-e_{467},$$
$$e_1 \partial e_{257}=e_{125}-e_{127}+e_{157}, e_3 \partial e_{257}=-e_{235}+e_{237}+e_{357},$$
$$e_4 \partial e_{257}=-e_{245}+e_{247}+e_{457}, e_6 \partial e_{257}=e_{256}+e_{267}-e_{567}.$$

Then an easy computation shows that in this case $\dim(\spann(F_3))=19$. 
\end{example}

\begin{remark}\label{rem:dimf3all} Similarly to the previous examples, we can directly compute $\dim(\spann(F_3))$ for several sign graph. In particular, if we consider $D_3, K_4$ and $K_3^3$, then $\dim(\spann(F_3))=10$. 
\end{remark}

\begin{lemma} $\dim(\spann(F_3^2))=10k_4$, $\dim(\spann(F_3^3))=10d_3$, $\dim(\spann(F_3^4))=10g_\circ$, $\dim(\spann(F_3^5))=19d_{3,1}$ and $\dim(\spann(F_3^6))=10k_{3,3}$.
\end{lemma}
\begin{proof} Assume that in the sign graph $G$ there are exactly $g_\circ=p$ distinct subgraphs isomorphic to a $\overline{G_\circ}$, $G_1,\dots, G_p$, none of which is a subgraph of a graph isomorphic to $D_3^1$.  Consider
$$F^4_{3,i}:=\{e_t\partial e_{ijk}~|~e_{ijk}\in C'_3,t\in[n]\setminus\{i,j,k\}, i,j,k \in G_i\}.$$
Since four edges in the graph $G$ can not appear in two distinct $\overline{G_\circ}$ at the same time, then none of the terms of the element $e_t\partial e_{ijk}\in F^2_{3,i}$ appear in the elements of $F_3^4\setminus F^4_{3,i}$. This shows that
$$\spann(F^4_3)=\bigoplus_{i=1}^p \spann(F^4_{3,i}).$$
By Corollary \ref{corol:signequivsameinvar} and Example \ref{ex:G0comput}, we have that $\dim(\spann(F^4_{3,i}))=10$ for all $i=1,\dots, p.$ This then implies that
$$\dim(\spann(F^4_3))=\sum_{i=1}^p \dim(\spann(F^4_{3,i}))=10g_\circ.$$

Using Remark \ref{rem:dimf3all} and Example \ref{exd31comput}, the same exact argument used in this case will prove the other equalities.
\end{proof}

\begin{lemma}\label{lemm:dimI32} For a signed graphic arrangement associated to a signed graph $G$ not containing a subgraph isomorphic to $B_2$, we have
$$\dim(I^3_2)=(n-2)(k_3+d_{2,1}+k_{3,3})-2k_4-2d_3-2g_\circ-2k_{3,3}-5d_{3,1}.$$
\end{lemma}
\begin{proof} By the previous lemmas
$$\dim(\spann(F_3))=\sum_{i=1}^6 \dim(\spann(F^i_3))=$$
$$=[(n-3)(k_3+d_{2,1}+k_{3,3})-12k_4-12d_3-12g_\circ-12k_{3,3}-24d_{3,1}]+ $$
$$ +10k_4+10d_3+10g_\circ+10k_{3,3}+19d_{3,1}=$$
$$ (n-3)(k_3+d_{2,1}+k_{3,3})-2k_4-2d_3-2g_\circ-2k_{3,3}-5d_{3,1}.$$
The thesis follows from the equality
$$\dim(I^3_2) = k_3+d_{2,1}+k_{2,2}+\dim(\spann(F_3)).$$
\end{proof}

\begin{proof}[Proof of Theorem \ref{theo:ourmain}]
By Remark \ref{rem:falkinvariantreduct} and Lemma \ref{lem:dimA2} we have
$$\phi_3=2\binom{n+1}{3}-n(\binom{n}{2}-k_3-d_{2,1}-k_{2,2})+\binom{n}{3}-\dim(I^3_2).$$
Because $2\binom{n+1}{3}-n\binom{n}{2} +\binom{n}{3}=0$, then from Lemma \ref{lemm:dimI32} we obtain
$$\phi_3=2(k_3+k_4+d_3+d_{2,1}+k_{2,2}+k_{3,3}+g_\circ)+5d_{3,1}.$$
\end{proof}

Let us see how our formula works on a non-trivial example.
\begin{example} We want to compute $\phi_3$ for the arrangement associated to the graph $G$ of Figure \ref{fig:finalexample}.

\begin{figure}[htbp]
\begin{tikzpicture}[baseline=10pt]
\draw (0,3) node[l, label=above:{11}](1){};
\draw (0,0) node[v](2){};
\draw (3,0) node[v](3){};
\draw (3,3) node[v](4){};
\draw (-0.2,1.5) node {7};
\draw (0.2,1.5) node {1};
\draw (1.5,2.7) node {2};
\draw (1.5,3.3) node {8};
\draw (2.8,1.5) node {3};
\draw (1.5,0.3) node {4};
\draw (1.5,-0.3) node {10};
\draw (0.8,1.2) node {5};
\draw (2.3,1.2) node {9};
\draw (1.8,0.8) node {6};
\draw[double, thick] (1)--(2);
\draw[double, thick] (1)--(3);
\draw[double, thick] (1)--(4);
\draw[double, thick] (2)--(3);
\draw[] (2)--(4);
\draw[] (3)--(4);
\end{tikzpicture}
\caption{The sign graph $G$}
\label{fig:finalexample}
\end{figure}
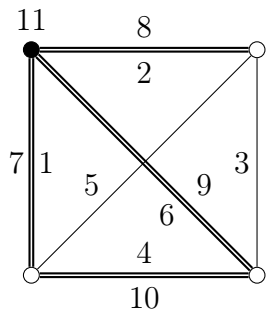
In this situation we have $E^+=\{1,2,3,4,5,6\}, E^-=\{7,8,9,10\}$ and $L=\{11\}$. In order to compute $\phi_3$ with the formula \eqref{eq:ourmainformula}, we need to compute the following:
\begin{itemize}
\item $k_3=|\{\{1,2,5\},\{1,4,6\},\{2,3,6\},\{3,4,5\},\{1,9,10\},\{6,7,9\},\{4,7,9\},$ $\{3,8,9\},\{5,7,8\}\}=9;$
\item $k_4=|\{\{1,2,3,4,5,6\},\{3,4,5,7,8,9\}\}|=2;$
\item $d_3=0;$
\item $d_{2,1}=|\{\{1,7,11\},\{6,9,11\},\{2,8,11\}\}|=3;$
\item $k_{2,2}=0;$
\item $k_{3,3}=0;$
\item $g_\circ=|\{\{1,2,5,7,8,11\},\{2,3,6,8,9,11\}\}|=2;$
\item $d_{3,1}=|\{\{1,4,6,7,9,10,11\}\}|=1.$
\end{itemize} 
From formula \eqref{eq:ourmainformula}, we obtain $$\phi_3=2(9+2+0+3+0+0+2)+5=37.$$ Notice that if we would try to compute the dimension of $F_3$ directly, we would have to write $96$ equations in the $e_{ijk}$.
\end{example}

\paragraph{\textbf{Acknowledgements}} The authors thank Professor Yoshinaga for the valuable discussions. The second authors also thanks Doctor Suyama and Doctor Tsujie for the valuable discussions on signed graphs.

During the preparation of this paper the second author was supported by the MEXT grant for Tenure Tracking system.

\bibliography{Bibliofile}{}
\bibliographystyle{plain}

\end{document}